\documentclass[11pt]{article}
\usepackage{geometry}
\usepackage{amsmath, amssymb, amsthm}
\usepackage{graphicx}
\usepackage{bm}
\usepackage[linesnumbered,ruled]{algorithm2e}
\usepackage{fullpage}
\newtheorem{theorem}{Theorem}
\newtheorem{lemma}{Lemma}
\title{A Note on Robust Subsets of Transversal Matroids}
\author{Naoyuki Kamiyama%
\thanks{This work was supported by JSPS KAKENHI Grant Number JP20K11680.}
}
\date{\small Institute of Mathematics for Industry, Kyushu University, Fukuoka, Japan.\\
{\ttfamily kamiyama@imi.kyushu-u.ac.jp}}

\begin{document}

\maketitle

\begin{abstract}
Robust subsets of matroids
were introduced by 
Huang and Sellier to 
propose approximate kernels for the 
matroid-constrained maximum vertex cover problem. 
In this paper, we prove that the bound for 
robust subsets of transversal matroids
given by Huang and Sellier can be improved. 
\end{abstract} 

\section{Introduction}

Robust subsets of matroids
were introduced by 
Huang and Sellier~\cite{HS22} to 
propose approximate kernels for the 
matroid-constrained maximum 
vertex cover problem 
(see Section~\ref{section:preliminaries} for the formal definition of 
a robust subset of a matroid).
By using this concept, 
Huang and Sellier~\cite{HS22} extended 
the approach proposed by Manurangsi~\cite{M19} for 
uniform matroids
to partition matroids, laminar matroids, and 
transversal matroids.

In this paper, we especially focus on robust subsets of 
transversal matroids.
We prove that the bound for 
robust subsets of transversal matroid given by 
Huang and Sellier~\cite{HS22} can be improved
(see Section~\ref{section:preliminaries} for the formal description 
of our result). 

\section{Preliminaries} 
\label{section:preliminaries} 

Throughout this paper, 
let $\mathbb{Z}_+$ denote the set of non-negative integers. 
For each positive integer $z$, 
we define $[z] := \{1,2,\dots,z\}$. 
In addition, we define $[0] := \emptyset$. 
For each finite set $X$ and each element $x$, 
we define $X + x := X \cup \{x\}$ and 
$X - x := X \setminus \{x\}$. 
Furthermore, for each finite set $X$, 
each function $\rho \colon X \to \mathbb{Z}_+$, and 
each subset $Y \subseteq X$, 
we define $\rho(Y) := \sum_{x \in Y}\rho(x)$. 

In this paper, we assume that 
every undirected graph is finite and simple. 
Thus, every edge of an undirected graph $G$
can be regraded as a set consisting of two distinct vertices of $G$. 
Let $G = (U,V;E)$ denote an undirected bipartite 
graph where the vertex set of $G$ is partitioned into 
$U$ and $V$, $E$ is the edge set of $G$, 
and each edge of $E$ consists of one vertex in $U$ 
and one vertex in $V$.
For every undirected bipartite 
graph $G = (U,V;E)$ and every edge $\{u,v\} \in E$, 
we assume that $u \in U$ and $v \in V$.
For each undirected bipartite 
graph $G = (U,V;E)$, 
we call
a subset $M \subseteq E$ 
a \emph{matching in $G$} if
$e \cap f = \emptyset$ 
for every pair of distinct edges $e,f \in M$.
For each undirected bipartite graph $G = (U,V;E)$
and each subset $F \subseteq E$, we define 
$\partial(F)$ as 
the set of vertices $v \in U \cup V$ 
such that there exists 
an edge $e \in F$ satisfying the condition 
that $v \in e$.

An ordered pair ${\bf M} = (U, \mathcal{I})$ of a finite 
set $U$ and a non-empty family $\mathcal{I}$ of 
subsets of $U$ is called 
a \emph{matroid} if, for every pair of subsets 
$I,J \subseteq U$, the following conditions are satisfied. 
\begin{description}
\setlength{\parskip}{0mm} 
\setlength{\itemsep}{1mm} 
\item[(I1)]
If $I \subseteq J$ and $J \in \mathcal{I}$, then 
$I \in \mathcal{I}$. 
\item[(I2)]
If $I,J \in \mathcal{I}$ and 
$|I| < |J|$, then there exists an element 
$u \in J \setminus I$ such that 
$I + u \in \mathcal{I}$. 
\end{description} 

Let ${\bf M} = (U,\mathcal{I})$ be a matroid.
Define ${\rm rk}({\bf M}) := \max_{I \in \mathcal{I}}|I|$.   
We call an element $B \in \mathcal{I}$ a \emph{base of ${\bf M}$} if 
$|B| = {\rm rk}({\bf M})$. 
For 
each positive integer $\ell$, we define 
the ordered pair ${\bf M}^{\ell} = (U, \mathcal{I}^{\ell})$ as follows. 
Define $\mathcal{I}^{\ell}$ as the family of subsets 
$I \subseteq U$ such that there exist pairwise disjoint subsets 
$I_1,I_2,\dots,I_{\ell} \subseteq I$ 
satisfying the 
following conditions. 
\begin{description}
\setlength{\parskip}{0mm} 
\setlength{\itemsep}{1mm} 
\item[(U1)]
$I_1 \cup I_2 \cup \cdots \cup I_{\ell} = I$. 
\item[(U2)]
For every integer $i \in [{\ell}]$, 
we have $I_i \in \mathcal{I}$. 
\end{description}
It is known that 
the ordered pair ${\bf M}^{\ell}$ is a matroid 
(see, e.g., \cite[Section~11.3]{O11}). 
For 
each function $\omega \colon U \to \mathbb{Z}_+$, 
we call a base of ${\bf M}$ an \emph{optimal base 
of {\bf M} with respect to $\omega$} if 
$\omega(B) \ge \omega(B^{\prime})$ holds
for every base $B^{\prime}$ of ${\bf M}$. 

In this paper, we especially focus on \emph{transversal matroids}. 
Let $G = (U,V;E)$ be an  
undirected bipartite 
graph. 
Then 
we define the matroid ${\bf M}_G = (U, \mathcal{I}_G)$ as follows. 
Define 
$\mathcal{I}_G$ as the family of subsets $I \subseteq U$ such that 
there exists a matching $M$ in $G$ satisfying the condition that  
$I = \partial(M) \cap U$. 
It is known that 
${\bf M}_G$ defined in this way is a matroid.
This kind of matroid is called a 
transversal matroid (see, e.g., \cite[Section~1.6]{O11}). 

For each matroid ${\bf M} = (U, \mathcal{I})$, 
each function $\omega \colon U \to \mathbb{Z}_+$, 
and 
each positive integer $k$, 
we call a subset $X \subseteq U$ 
a \emph{$k$-robust subset of ${\bf M}$ with respect to $\omega$} 
if, for every base $B$ of ${\bf M}$, 
there exist 
pairwise disjoint 
subsets $X_1,X_2,\dots,X_{|B\setminus X|} \subseteq X \setminus B$ and 
a bijection 
$\phi \colon B \setminus X \to [|B \setminus X|]$ 
satisfying the following conditions. 
\begin{description}
\setlength{\parskip}{0mm} 
\setlength{\itemsep}{1mm} 
\item[(R1)]
For every integer $i \in [|B \setminus X|]$, 
we have $|X_i| = k$. 
\item[(R2)]
For every element $u \in B \setminus X$ and every element 
$v \in X_{\phi(u)}$, we have 
$\omega(u) \le \omega(v)$. 
\item[(R3)] 
For every element $(u_1,u_2,\dots,u_{|B\setminus X|}) 
\in X_1 \times X_2 \times \dots \times X_{|B \setminus X|}$, 
we have 
\begin{equation*}
(B \cap X) \cup \{u_1,u_2,\dots,u_{|B \setminus X|}\} \in 
\mathcal{I}. 
\end{equation*}
\end{description}
Then for 
each matroid ${\bf M} = (U, \mathcal{I})$, 
each function $\omega \colon U \to \mathbb{Z}_+$, 
and 
each positive integer $k$,
we define $\tau({\bf M}, \omega, k)$ as the minimum positive integer 
$\ell$ such that 
every optimal base of ${\bf M}^{\ell}$ with respect to 
$\omega$ is a $k$-robust subset of ${\bf M}$ 
with respect to $\omega$. 
(See also Section~\ref{section:definition} for remarks on 
the definition of $\tau({\bf M}, \omega, k)$.)

Huang and Sellier~\cite{HS22} proved 
that, for every undirected bipartite graph $G = (U,V;E)$, 
every function $\omega \colon U \to \mathbb{Z}_+$, 
and every positive integer $k$, we have 
\begin{equation*}
\tau({\bf M}_G, \omega, k) \le k + {\rm rk}({\bf M}_G) - 1.
\end{equation*} 
In this paper, we prove the 
following theorem. 
We prove Theorem~\ref{main_theorem} in Section~\ref{section:proof}. 

\begin{theorem} \label{main_theorem} 
For every undirected bipartite graph $G = (U,V;E)$, 
every function $\omega \colon U \to \mathbb{Z}_+$, 
and every positive integer $k$, we have 
\begin{equation*}
\tau({\bf M}_G, \omega, k) \le k.
\end{equation*} 
\end{theorem} 

\section{Auxiliary Bipartite Graphs} 
\label{section:auxiliary} 

Throughout this section, 
let $G = (U,V;E)$ be an undirected bipartite graph. 
Furthermore, we are given 
a function $\omega \colon U \to \mathbb{Z}_+$
and a positive integer $k$. 

We define the undirected bipartite graph $G^k = (U,V^k;E^k)$
as follows. 
For each vertex $v \in V$ and 
each integer $t \in [k]$,
we prepare a new vertex $v(t)$. 
Then we define 
$V^k := 
\{v(t) \mid v \in V, t \in [k]\}$. 
For each vertex 
$u \in U$
and each vertex 
$v(t) \in V^k$,
$\{u,v(t)\} \in E^k$ if and 
only if $\{u,v\} \in E$. 
We define the function 
$\overline{\omega} \colon E^k \to \mathbb{Z}_+$ 
by 
$\overline{\omega}(\{u,v(t)\}) := \omega(u)$
for each edge $\{u,v(t)\} \in E^k$.
We call a matching $N$ in $G^k$ a 
\emph{maximum-weight matching in $G^k$ with respect to $\overline{\omega}$} if 
$\overline{\omega}(N) \ge \overline{\omega}(N^{\prime})$ for every 
matching $N^{\prime}$ in $G^k$. 

\begin{lemma} \label{lemma:independent}
Let $I$ be a subset of $U$.
Then $I \in \mathcal{I}_G^k$ if and only if 
there exists a matching $N$ in $G^k$ 
such that 
$I = \partial(N) \cap U$. 
\end{lemma}
\begin{proof}
Assume that $I \in \mathcal{I}_G^k$. 
Then there exist pairwise disjoint 
subsets $I_1,I_2, \dots,I_k \subseteq I$ 
satisfying (U1) and (U2). 
For every integer $t \in [k]$, 
since $I_t \in \mathcal{I}_G$, 
there exists a matching $M_t$ in $G$ such that 
$I_t = \partial(M_t) \cap U$. 
For each integer $t \in [k]$, we define 
$N_t := \{\{u,v(t)\} \mid \{u,v\} \in M_t\}$.
In addition, we define 
$N := N_1 \cup N_2 \cup \cdots \cup N_k$. 
Then $N$ is a matching in $G^k$ and 
$I = \partial(N) \cap U$. 

Next, we assume that 
there exists a matching $N$ in $G^k$ 
such that 
$I = \partial(N) \cap U$.
For each integer $t \in [k]$, 
we define $M_t := \{\{u,v\} \mid \{u,v(t)\} \in N\}$
and $I_t := \partial(M_t)$. 
Then for every integer $t \in [k]$, 
since $M_t$ is a matching in $G$, $I_t \in \mathcal{I}_G$.
Furthermore, 
$I_1 \cup I_2 \cup \dots \cup I_k = I$. 
\end{proof} 

\begin{lemma} \label{lemma:optimal}
For every optimal base $X$ of ${\bf M}_G^k$ with respect to $\omega$,
there exists a maximum-weight matching $N$ in $G^k$ with respect to $\overline{\omega}$ 
such that 
$X = \partial(N) \cap U$. 
\end{lemma}
\begin{proof}
Let $X$ be an optimal base of ${\bf M}_G^k$ with respect to $\omega$. 
Then it follows from Lemma~\ref{lemma:independent} that 
there exists a matching $N$ in $G^k$ such that 
$X = \partial(N) \cap U$. 
Assume that $N$ is not a maximum-weight matching in $G^k$
with respect to $\overline{\omega}$. 
In this case, there exists a matching $N^{\prime}$ in $G^k$ 
such 
that $\overline{\omega}(N) < \overline{\omega}(N^{\prime})$. 
Define $Y := \partial(N^{\prime}) \cap U$. 
Lemma~\ref{lemma:independent} implies that 
$Y \in \mathcal{I}_G^k$.
Furthermore, the definition of $\overline{\omega}$ 
implies that 
$\omega(X) = \overline{\omega}(N) < \overline{\omega}(N^{\prime}) = \omega(Y)$. 
Since (I2) implies that 
there exists a base $X^{\prime}$ 
of ${\bf M}_G^k$ such that 
$\omega(Y) \le \omega(X^{\prime})$. 
This contradicts the fact that 
$X$ is an optimal base of ${\bf M}_G^k$ with respect to $\omega$. 
This completes the proof. 
\end{proof} 

In what follows, let $N$ be a matching in $G^k$.

Let $P = 
(x_1,y_1,x_2,y_2,\dots,y_{\ell},x_{\ell+1})$
be a sequence 
of distinct vertices of $G^k$. 
Then $P$ is called an 
\emph{alternating path in $G^k$ 
with respect to $N$} if 
the following conditions are satisfied. 
\begin{itemize}
\setlength{\parskip}{0mm} 
\setlength{\itemsep}{1mm} 
\item
For every integer $i \in [\ell+1]$ (resp.\ $i \in [\ell]$), 
$x_i \in U$ (resp.\ and $y_i \in V^k$). 
\item
For every integer $i \in [\ell]$ 
$\{x_i,y_i\} \in E^k \setminus N$ and $\{x_{i+1},y_i\} \in N$. 
\item
$x_1 \notin \partial(N)$. 
\end{itemize}
We also say that 
$P$ is an alternating path in $G^k$ 
with respect to $N$ form $x_1$ to $x_{\ell+1}$.
Notice that
$\overline{\omega}(\{x_{i+1},y_i\}) = \overline{\omega}(\{x_{i+1},y_{i+1}\})$
for every integer $i \in [\ell -1]$.
Assume that 
$P$ is an alternating path in $G^k$ 
with respect to $N$.
Then we define the new matching $N \oplus P$ in $G^k$ by 
\begin{equation*}
N \oplus P := 
(N \setminus \big\{\{x_2,y_1\},\{x_3,y_2\},\dots,\{x_{\ell+1},y_{\ell}\}\big\}) 
\cup 
\big\{\{x_1,y_1\},\{x_2,y_2\},\dots,\{x_{\ell},y_{\ell}\}\big\}. 
\end{equation*}
Furthermore, we define $\overline{\omega}(P)$ by   
\begin{equation*}
\begin{split}
\overline{\omega}(P) & := 
\sum_{i \in [\ell]} \overline{\omega}(\{x_i,y_i\}) 
- 
\sum_{i \in [\ell]} \overline{\omega}(\{x_{i+1},y_i\})
= \omega(x_1) - \omega(x_{\ell+1}). 
\end{split} 
\end{equation*}
Notice that 
$\overline{\omega}(N\oplus P) = \overline{\omega}(N) + \overline{\omega}(P)$. 
Thus, if $N$ is a maximum-weight matching in $G^k$ with respect to $\omega$, 
then 
there does not exist an alternating path $P$ in $G^k$ with respect to $N$ 
such that $\overline{\omega}(P) > 0$. 

Let $P = 
(x_1,y_1,x_2,y_2,\dots,x_{\ell},y_{\ell})$
be a sequence 
of distinct vertices of $G^k$. 
Then $P$ is called an 
\emph{augmenting path in $G^k$ 
with respect to $N$} if 
the following conditions are satisfied. 
\begin{itemize}
\setlength{\parskip}{0mm} 
\setlength{\itemsep}{1mm} 
\item
$(x_1,y_1,x_2,y_2,\dots,y_{\ell-1},x_{\ell})$ 
is an alternating path in $G^k$ 
with respect to $N$. 
\item
$y_{\ell} \in V^k$ and 
$y_{\ell} \notin \partial(N)$.
\end{itemize}
We also say that 
$P$ is an augmenting path in $G^k$ 
with respect to $N$ form $x_1$ to $y_{\ell}$.
If $\partial(N) \cap U$ is a base of ${\bf M}_G^k$, then 
there does not exist an augmenting path in $G^k$ with respect to $N$. 

\section{Proof of Theorem~\ref{main_theorem}} 
\label{section:proof} 

In this section, we prove Theorem~\ref{main_theorem}. 
Throughout this section, 
let $G = (U,V;E)$ be an undirected bipartite graph. 
Furthermore, we are given 
a function $\omega \colon U \to \mathbb{Z}_+$
and a positive integer $k$. 
We define the function $\overline{\omega} \colon E^k \to \mathbb{Z}_+$
as in Section~\ref{section:auxiliary}.

In what follows, let $X$ be an optimal base of ${\bf M}_G^k$ with respect to 
$\omega$. 
Furthermore, 
let $B$ be a base of ${\bf M}_G$. 
Lemma~\ref{lemma:optimal} implies that 
there exists a maximum-weight matching $N$ in $G^k$ with respect to 
$\overline{\omega}$ such that 
$X = \partial(N) \cap U$. 
In addition, since $B \in \mathcal{I}$, 
there exists a matching $M$ in $G$ such that 
$B = \partial(M) \cap U$. 
For each vertex $v \in V$, 
we define $X_v$ as the set of 
vertices $u \in X$
such that 
there exists an integer $t \in [k]$ 
satisfying the condition that 
$\{u,v(t)\} \in N$. 
Define the function $\varphi \colon B \to V$ by 
defining $\varphi(u)$ as the vertex $v \in V$ 
such that 
$\{u,v\} \in M$. 

Define $W$ as the set of vertices $v \in V$ such that 
$\varphi(u) \neq v$ holds for every vertex $u \in X_v \cap B$. 
We define the simple directed graph $D = (W,A)$ as follows. 
For each pair of distinct vertices $v,v^{\prime} \in W$, 
the arc $(v,v^{\prime})$ from $v$ to $v^{\prime}$ exists in $A$ if and only if 
there exists a vertex $u \in X_v \cap B$ such that 
$\varphi(u) = v^{\prime}$. 
For each arc $(v,v^{\prime}) \in A$, 
we say that 
\emph{$(v,v^{\prime})$ enters $v^{\prime}$}
(resp.\ \emph{leaves $v$}). 
For each vertex $v \in W$, 
$v$ is called a \emph{source} 
(resp.\ \emph{sink}) \emph{of $D$} if 
there does not exist an arc in $A$
that enters (resp.\ \emph{leaves}) $v$. 

\begin{lemma} \label{lemma:in_degree}
For every vertex $v \in W$, 
the number of arcs in $A$ that enter $v$ is 
at most one.
\end{lemma}
\begin{proof}
Assume that there exists a vertex $v \in W$ 
such that 
distinct arcs $(p,v), (q,v) \in A$ enter $v$. 
Then there exist distinct vertices 
$u_p \in X_p \cap B$ and 
$u_q \in X_q \cap B$ such that 
$\{u_p,v\}, \{u_q,v\} \in M$. 
However, this contradicts the fact that 
$M$ is a matching in $G$. 
\end{proof}

\begin{lemma} \label{lemma:source} 
For every vertex $u \in B \setminus X$, 
$\varphi(u) \in W$ and $\varphi(u)$ is a source of $D$.
\end{lemma}
\begin{proof}
Let $u$ be a vertex in $B \setminus X$. 

Assume that $\varphi(u) \notin W$. 
Then there exists a vertex $w \in X_{\varphi(u)}$ such that 
$\varphi(w) = \varphi(u)$. 
Since $u \neq w$ and $\{u,\varphi(u)\} \in M$, this contradicts the fact that 
$M$ is a matching in $G$. 

Assume that there exists an arc $(v,\varphi(u)) \in A$ that enters $\varphi(u)$. 
In this case, there exists a vertex $w \in X_v \cap B$ such that 
$\{w,\varphi(u)\} \in M$.  
Since $u \neq w$ and 
$\{u,\varphi(u)\} \in M$, 
this contradicts the fact that $M$ is a matching in $G$. 
\end{proof} 

\begin{lemma} \label{lemma:w} 
For every vertex $v \in W$ and 
every vertex $u \in X_v \cap B$, 
we have $\varphi(u) \in W$. 
\end{lemma}
\begin{proof}
Assume there exist a vertex 
$v \in W$ and a vertex
$u \in X_v \cap B$
such that $\varphi(u) \notin W$. 
Then there 
exists a vertex $w \in X_{\varphi(u)}$ 
such that $\varphi(w) = \varphi(u)$. 
Since $u \neq w$ follows from the fact that 
$\varphi(u) \neq v$, this contradicts the fact that 
$M$ is a matching in $G$. 
\end{proof} 

\begin{lemma} \label{lemma:sink}
For every sink $v$ of $D$, 
we have $X_v \cap B = \emptyset$. 
\end{lemma}
\begin{proof}
Let $v$ be a sink of $D$. 
Assume that $X_v \cap B \neq \emptyset$. 
Let $u$ be a vertex in $X_v \cap B$. 
Lemma~\ref{lemma:w} implies that 
$\varphi(u) \in W$. 
Furthermore, since $v \in W$, 
$\varphi(u) \neq v$. 
This implies that 
the arc $(v,\varphi(u))$ leaves $v$. 
This contradicts the fact that $v$ is a sink of $D$. 
\end{proof} 

It is not difficult to see that 
Lemmas~\ref{lemma:in_degree} and \ref{lemma:source} imply that, for 
each vertex $u \in B \setminus X$, there exists a sequence 
$P_u = (z^u_1,z^u_2,\dots,z^u_{\ell(u)})$
of vertices in $V$ satisfying the following conditions. 
\begin{itemize}
\setlength{\parskip}{0mm} 
\setlength{\itemsep}{1mm} 
\item 
For every vertex $u \in B \setminus X$, 
$z^u_1 = \varphi(u)$.
\item
For every vertex $u \in B \setminus X$, 
$z^u_{\ell(u)}$ is a sink of $D$. 
\item 
For every vertex $u \in B \setminus X$, 
and every integer $i \in [\ell(u)-1]$, 
$(z^u_i,z^u_{i+1}) \in A$. 
\item
For every pair of distinct vertices $u,w \in B \setminus X$ and 
every pair of distinct integers $i \in [\ell(u)]$ and 
$j \in [\ell(w)]$, 
we have 
$z^{u}_i \neq z^{w}_i$. 
\end{itemize}
For each vertex $u \in B \setminus X$, 
we define ${\sf p}_u := z^u_{\ell(u)-1}$
and ${\sf q}_u := z^u_{\ell(u)}$. 

Let $\sigma$ be an arbitrary function 
from $[|B \setminus X|]$
to $\{{\sf q}_u\mid u \in B \setminus X\}$. 
Then for each integer $i \in [|B \setminus X|]$, 
we define $X_i := X_{\sigma(i)}$. 
In addition, we define the function $\phi \colon B \setminus X \to [|B \setminus X|]$
as follows. 
For each vertex $u \in B \setminus X$ and 
each integer $i \in [|B \setminus X|]$,
$\phi(u) = i$ if and only if 
$\sigma(i) = {\sf q}_u$. 
Since Lemma~\ref{lemma:sink} implies that
$X_v \cap B = \emptyset$ for every sink $v$ of $D$,  
$X_i \subseteq X \setminus B$ for every integer $i \in [|B\setminus X|]$. 
In the rest of this section, 
we prove that 
$X_1,X_2,\dots,X_{|B\setminus X|}$ and $\phi$ 
satisfy (R1), (R2), and (R3).
This completes the proof of Theorem~\ref{main_theorem}.

\begin{lemma} \label{lemma:alternating} 
Let $u$ be a vertex in $B \setminus X$.
Then for every vertex $w \in X_{{\sf p}_u} \cup X_{{\sf q}_u}$,
there exists an alternating path in $G^k$ with respect to $N$ 
from $u$ to $w$. 
\end{lemma}
\begin{proof}
For each integer $i \in [\ell(u)]$, we define 
$y_i := z_i^u$. 
We prove that, for every integer $i \in [\ell(u)]$ and 
every vertex $w \in X_{y_i}$,
there there exists an alternating path in $G^k$ with respect to $N$ 
from $u$ to $w$ by induction on $i$. 
Notice that since $u \notin X$, $u \notin \partial(N)$. 

Since $u \notin X$, 
$\{u,y_1(t)\} \notin N$ holds 
for every integer $t \in [k]$. 
Thus, for every vertex $w \in X_1$, since 
there exists an integer $t \in [k]$ such that 
$\{w,y_1(t)\} \in N$, 
there there exists an alternating path in $G^k$ with respect to $N$ 
from $u$ to $w$. 

Assume that the statement holds 
for some integer $i \in [\ell(u)-1]$. 
Since $y_i$ is not a sink of $D$, 
there exists a vertex $x_i \in X_{y_i} \cap B$
such that $\varphi(x_i) = y_{i+1}$. 
Then 
$\{x_i,y_{i+1}(t)\} \in E^k$ 
for every integer $t \in [k]$
(see Figure~\ref{fig:alternating}). 
The induction hypothesis implies that there 
exists an alternating path in $G^k$ with respect to $N$ 
from $u$ to $x_i$.
This implies that, 
for every vertex $w \in X_{y_{i+1}}$, 
there 
exists an alternating path in $G^k$ with respect to $N$ 
from $u$ to $w$.
This completes the proof. 
\end{proof} 

\begin{figure}[h]
\begin{center}
\includegraphics[width=12cm]{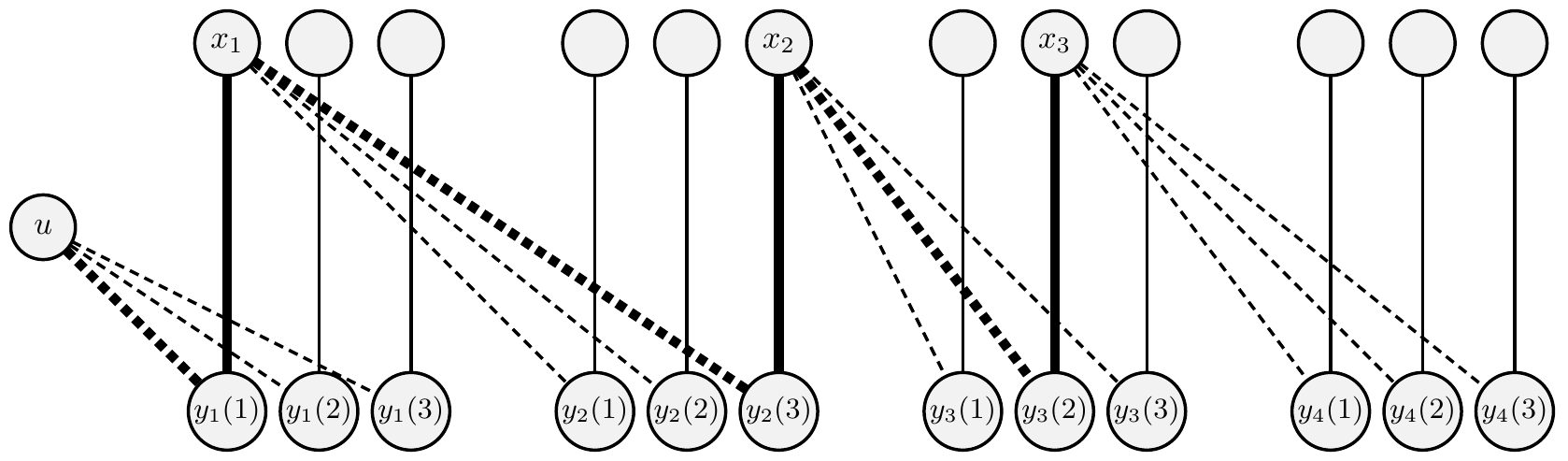}
\end{center}
\caption{An illustration of an alternating path.
The solid lines are edges in $N$, and 
the broken lines are edges in $E^k \setminus N$.}
\label{fig:alternating}
\end{figure}

\begin{lemma} \label{lemma:R1} 
For every vertex $u \in B \setminus X$, 
we have 
$\{{\sf q}_u(t) \mid t \in [k]\} \subseteq \partial(N)$. 
\end{lemma}
\begin{proof}
Let $w$ be a vertex in $X_{{\sf p}_u} \cap B$ such that 
$\varphi(w) = {\sf q}_u$. 
Lemma~\ref{lemma:alternating} implies that 
there exists an alternating path in $G^k$ with respect to $N$ 
from $u$ to $w$. 
Assume that there exists an integer $t \in [k]$ such that 
${\sf q}_u(t) \notin \partial(N)$. 
Since $\{w,{\sf q}_u(t)\} \in E^k \setminus N$, there exists an augmenting 
path in $G^k$ with respect to $N$. 
This contradicts the fact that $X$ is a base of ${\bf M}_G^k$. 
\end{proof} 
Lemma~\ref{lemma:R1} implies that 
$X_1,X_2,\dots,X_{|B\setminus X|}$ and $\phi$ 
satisfy (R1). 

\begin{lemma} \label{lemma:R2} 
For every vertex $u \in B \setminus X$
and every vertex $w \in X_{{\sf q}_u}$, 
we have $\omega(u) \le \omega(w)$. 
\end{lemma}
\begin{proof}
For every vertex in $u \in B \setminus X$, 
if there exists a vertex $w \in X_{{\sf q}_u}$
such that $\omega(u) > \omega(w)$, 
then it follows from Lemma~\ref{lemma:alternating} that there exists 
an alternating path $P$ in $G^k$ with respect to $N$ 
such that $\overline{\omega}(P) > 0$.
However, this contradicts the fact that 
$N$ is a maximum-weight matching in $G^k$ with respect to 
$\overline{\omega}$. 
\end{proof} 
Lemma~\ref{lemma:R2} implies that 
$X_1,X_2,\dots,X_{|B\setminus X|}$ and $\phi$ 
satisfy (R2). 

\begin{lemma} \label{lemma:R3} 
For every element $(u_1,u_2,\dots,u_{|B\setminus X|}) 
\in X_1 \times X_2 \times \dots \times X_{|B \setminus X|}$, 
we have 
\begin{equation*}
(B \cap X) \cup \{u_1,u_2,\dots,u_{|B \setminus X|}\} \in 
\mathcal{I}_G. 
\end{equation*}
\end{lemma}
\begin{proof}
For each vertex $u \in B \setminus X$, let $w_u$ be an 
arbitrary vertex in $X_{\phi(u)}$ ($=X_{{\sf q}_u}$). 
For every vertex in $u \in B \setminus X$,
Lemma~\ref{lemma:alternating} implies that 
there exists an alternating path 
\begin{equation*}
(x^u_{0},z^u_1(t^u_1),x^u_{1},z^u_2(t^u_2),
\dots, x^u_{\ell(u)-1}, z^u_{\ell(u)}(t^u_{\ell(u)}),x^u_{\ell(u)}) 
\end{equation*}
in $G^k$ with respect to $N$ from $u$ to $w_u$.
Notice that
$x^u_{0} = u$, $x^u_{\ell(u)} = w_u$, and 
$\{x^u_{i-1},z^u_{i}\} \in M$ 
for every integer $i \in [\ell(u)]$. 
Define $M^{\prime}$ by 
\begin{equation*}
M^{\prime} := 
\big(M \setminus 
\bigcup_{u \in B \setminus X}\big\{
\{x^u_{i-1},z^u_{i}\} \mid i \in [\ell(u)]
\big\}
\big)
\cup
\bigcup_{u \in B \setminus X}\big\{
\{x^u_i,z^u_i\} \mid i \in [\ell(u)]
\big\}.
\end{equation*}
Then $M^{\prime}$ is a matching in $G$, and 
\begin{equation*}
\partial(M^{\prime}) \cap U = 
(B \cap X) \cup \{w_u \mid u \in B \setminus X\}. 
\end{equation*} 
This completes the proof. 
\end{proof} 

Lemma~\ref{lemma:R3} implies that 
$X_1,X_2,\dots,X_{|B\setminus X|}$ and $\phi$ 
satisfy (R3).
This completes the proof. 

\section{Remarks on the Definition of $\tau$}
\label{section:definition} 

Throughout this section, 
let ${\bf M} = (U,\mathcal{I})$ be a matroid, and 
we are given 
a function 
$\omega \colon U \to \mathbb{Z}_+$
and a positive integer $k$. 
Define $n := |U|$.

Let $\mathcal{F}$ be the set of bijections 
$\xi \colon [n] \to U$ such that, for 
every pair of integers $i,j \in [n]$ such that 
$i \le j$, we have $\omega(\xi(i)) \ge \omega(\xi(j))$. 
Then for each positive integer $\ell$ and 
each element $\xi \in \mathcal{F}$, we define 
${\sf Greedy}_{\ell}(\xi)$ as the output of Algorithm~\ref{alg:greedy}. 
\begin{algorithm}[h]
Define $I_0 := \emptyset$.\\
\For{$i = 1,2,\dots,n$}
{
  If $I_{i-1} + \xi(i) \in \mathcal{I}^{\ell}$, then we 
  $I_i := I_{i-1} + \xi(i)$. 
  Otherwise, we define $I_i := I_{i-1}$.
}
Output $I_n$, and halt. 
\caption{Algorithm for defining ${\sf Greedy}_{\ell}(\xi)$}
\label{alg:greedy}
\end{algorithm}

Huang and Sellier~\cite{HS22}
originally defined $\tau({\bf M},\omega,k)$ as 
the minimum positive integer 
$\ell$ such that, for every element $\xi \in \mathcal{F}$, 
${\sf Greedy}_{\ell}(\xi)$ is a $k$-robust subset of ${\bf M}$
with respect to $\omega$. 
In what follows, we prove that our definition is equivalent 
to the original definition in \cite{HS22}. 

In what follows, let $\ell$ be a positive integer. 

It is known that, 
for every element $\xi \in \mathcal{F}$, 
${\sf Greedy}_{\ell}(\xi)$ is an optimal base of ${\bf M}^{\ell}$ with respect to 
$\omega$~\cite{R57}.
Thus, what remains is to prove that, for every optimal base $Z$ 
of ${\bf M}^{\ell}$ with respect to $\omega$, there exists an element 
$\xi \in \mathcal{F}$ such that 
${\sf Greedy}_{\ell}(\xi) = Z$. 

Let $Z$ be an optimal base 
of ${\bf M}^{\ell}$ with respect to $\omega$. 
Define $s := |\{\omega(u) \mid u \in U\}|$, i.e., 
$s$ is the number of distinct values in $\{\omega(u) \mid u \in U\}$. 
Furthermore, we define $\lambda_1,\lambda_2,\dots,\lambda_s$ as 
the distinct non-negative integers such that 
$\lambda_1 > \lambda_2 > \dots > \lambda_s$
and 
$\{\lambda_1,\lambda_2,\dots,\lambda_s\} = 
\{\omega(u) \mid u \in U\}$. 
For each integer $i \in [s]$, we define 
$U_i$ as the set of elements $u \in U$ such that 
$\omega(u) \ge \lambda_i$. 
Furthermore, for each integer $i \in [s]$, we define 
$n_i := |U_i|$.
Define $\xi_Z$ as an element in $\mathcal{F}$ such that 
\begin{equation} \label{eq1:xi}
\{\xi_Z(n_{i-1}+1), \xi_Z(n_{i-1}+2), \dots, 
\xi_Z(n_{i-1} + |Z \cap (U_{i} \setminus U_{i-1})|)\} = 
Z \cap (U_{i} \setminus U_{i-1})
\end{equation}
for every integer $i \in [s]$, where we define $n_0 := 0$
and $U_0 := \emptyset$.
In what follows, we prove that 
${\sf Greedy}_{\ell}(\xi_Z) = Z$. 
This completes the proof of the equivalence. 

For each integer $i \in \{0\} \cup [s]$, 
we define $Z_i = Z \cap U_i$. 
Then we prove that, for every integer $i \in \{0\} \cup [s]$, 
${\sf Greedy}_{\ell}(\xi_Z) \cap U_i = Z_i$ 
by induction on $i$. 
If $i = 0$, the statement clearly holds. 
Let $i$ be an integer in $[s]$.
Assume that 
${\sf Greedy}_{\ell}(\xi_Z) \cap U_{i-1} = Z_{i-1}$.  
Since $Z \in \mathcal{I}^{\ell}$, 
(I1) implies 
that $Z_i \in \mathcal{I}^{\ell}$. 
Thus, since ${\sf Greedy}_{\ell}(\xi_Z) \cap U_{i-1} = Z_{i-1}$,
\eqref{eq1:xi} implies that 
$Z_i \subseteq {\sf Greedy}_{\ell}(\xi_Z) \cap U_i$.
What remains is to prove that
$Z_i + u \notin \mathcal{I}^{\ell}$ holds
for every element 
$u \in U_i \setminus (U_{i-1} \cup Z_i)$.
In order to prove this by contradiction, 
we assume that there exists an element $u \in U_i \setminus (U_{i-1} \cup Z_i)$
such that $Z_i + u \in \mathcal{I}^{\ell}$.  
Since $Z$ is a base of ${\bf M}^{\ell}$, 
$Z + u \notin \mathcal{I}^{\ell}$. 
In this case, it is well known that 
there exists a subset $C \subseteq Z + u$ such that
$C \notin \mathcal{I}^{\ell}$ and  
$Z + u - w \in \mathcal{I}^{\ell}$
for every element $w \in C$
(see, e.g., \cite[Corollary 1.2.6]{O11} and 
\cite[Excircise~5, p.20]{O11}). 
If $\omega(w) \ge \omega(u)$ for every element $w \in C$, 
then $C \subseteq Z_i + u$.
However, this contradicts the fact that $Z_i + u \in \mathcal{I}^{\ell}$. 
Thus, there exists an element $v \in C$ such that 
$\omega(u) > \omega(v)$. 
Since $Z + u - v$ is a base of ${\bf M}^{\ell}$,
this contradicts the fact that $Z$ is an optimal base of 
${\bf M}^{\ell}$ with respect to $\omega$. 
Thus, 
$Z_i + u \notin \mathcal{I}^{\ell}$
holds 
for every element 
$u \in U_i \setminus (U_{i-1} \cup Z_i)$. 
This completes the proof. 

\bibliographystyle{plain}% bib style
\bibliography{robust_matroid}

\end{document}